\documentclass[reqno]{amsart}

\title[Characterising path-, ray- and branch spaces, and end spaces of infinite graphs]{Characterising path-, ray- and branch spaces of order trees, and end spaces of infinite graphs}
\author{Max Pitz}
\address{Universit\"at Hamburg, Department of Mathematics, Bundesstrasse 55 (Geomatikum), 20146 Hamburg, Germany}
\email{max.pitz@uni-hamburg.de}
\usepackage{amsmath,amssymb,amsthm}  
\usepackage{mathtools}
\usepackage{tikz, tikz-cd}
\usetikzlibrary{shapes.geometric}
\usetikzlibrary{shapes,decorations.pathmorphing,backgrounds,positioning,fit,petri, hobby}
\usepackage{comment}
\usepackage{enumerate}
\usepackage{enumitem}
\usepackage{mathrsfs}

\usepackage{xcolor} 	
\usepackage[unicode]{hyperref}
\hypersetup{
	colorlinks,
    linkcolor={red!60!black},
    citecolor={green!60!black},
    urlcolor={blue!60!black},
}
\usepackage[abbrev, msc-links]{amsrefs} 

\usepackage[utf8]{inputenc}
\usepackage[T1]{fontenc}
\usepackage{lmodern}
\usepackage[babel]{microtype}
\usepackage[english]{babel}

\usepackage{geometry}
\usepackage{enumitem}

\let\polishlcross=\l
\def\l{\ifmmode\ell\else\polishlcross\fi}

\let\emptyset=\varnothing

\let\theta=\vartheta
\let\rho=x
\let\phi=\varphi

\def\NN{\mathbb N}
 
\def\QQ{\mathbb Q}

\def\cA{{\mathcal A}}

\def\cL{{\mathcal L}}

\def\cI{{\mathcal I}}
\def\cB{{\mathcal B}}
\def\cC{{\mathscr C}}

\def\cP{{\mathcal P}}

\def\cG{{\mathcal G}}
\def\cR{{\mathcal R}}
\def\cS{{\mathcal S}}

\def\cU{{\mathcal U}}

\defˆ{^}
\newcommand{\script}{\mathcal}
\newcommand{\parentheses}[1]{{\left( {#1} \right)}}

\newcommand{\p}{\parentheses}

\newcommand{\closure}[1]{\overline{#1}}
\newcommand{\Set}[1]{{\left\lbrace {#1} \right\rbrace}}

\newcommand{\pair}[1]{\langle {#1} \rangle}
\def\set#1:#2{\Set{{#1} \colon {#2}}}

\newcommand{\dc}[1]{\lceil #1\rceil}
\newcommand{\uc}[1]{\lfloor #1\rfloor}

\theoremstyle{plain}
\newtheorem{thm}{Theorem}[section]

\newtheorem{prop}[thm]{Proposition}

\newtheorem{cor}[thm]{Corollary}
\newtheorem{lemma}[thm]{Lemma}

\theoremstyle{definition}
\newtheorem{exmp}[thm]{Example}

\newtheorem{defn}[thm]{Definition}

\begin{document}

\begin{abstract}
We investigate path-, ray- and branch spaces of trees, certain topological spaces naturally associated with order theoretic trees, and provide topological characterisations for these spaces in terms of the existence of certain kinds of (sub-)bases. 
Our results yield a solution to a problem due to Diestel from 1992, by establishing the following equivalences for any Hausdorff space $X$:
	\begin{enumerate}
		\item $X$ is homeomorphic to the end space of a graph,
		\item $X$ is homeomorphic to the ray space of a special order tree, 
		\item $X$ admits a nested clopen subbase that is
 noetherian, hereditarily complete and $\sigma$-disjoint.
 	\end{enumerate}
\end{abstract}

\vspace*{-36pt}
\maketitle

\section{Introduction}

End spaces arise as the boundary of an infinite graph in a standard sense generalising the theory of the Freudenthal boundary developed by Freudenthal and Hopf in the 1940's for infinite groups \cite{freudenthal1942neuaufbau,freudenthal1944enden,hopf1943enden}. Although end spaces have been studied for decades by the graph-theoretic community in numerous papers by Halin, Jung, Polat, Diestel and others \cite{carmesin2019all, diestel2006end, diestel2003graph , Halin_Enden64 , halin1978simplicial, jung1968wurzelbaume, FirstSecondCountable, kurkofka2021approximating, pitz2020unified, polat1996ends, polat1996ends2, sprussel2008end}, we still do not have a characterisation of end spaces in topological terms, a problem formally posed by Diestel in 1992 \cite[Problem~5.1]{diestel1992end}. 
The purpose of the present paper is to provide such a topological characterisation, phrased in terms of the existence of a certain subbase. As a byproduct of our analysis, we also establish a number of further, new  topological properties shared by all end spaces, for example that every end space contains a dense, completely metrizable subspace.

Characterisations of topological spaces in terms of the existence of certain types of (sub-)bases belong to most useful results in general topology. The most familiar ones are \emph{Urysohn's metrization theorem} (regular spaces with a countable (sub-)base are metrizable), the \emph{Nagata-Smirnov-Bing metrization theorem} (regular spaces with a $\sigma$-discrete (sub-)base are metrizable), \emph{de Groot's ultrametrization theorem} (Hausdorff spaces with a $\sigma$-discrete clopen (sub-)base are ultrametrizable), and \emph{Alexander's subbase lemma} (spaces where every cover consisting of elements of some subbase has a finite subcover are compact). We refer the reader to the survey \cite{nyikos1975some} by Nyikos  for additional `surprising base properties in topology', and the paper \cite{van1979subbase} and the references therein for further examples of subbase characterisations.

All our results on end spaces of infinite graphs are obtained indirectly by investigating certain topologies naturally associated with order-theoretic trees as introduced by Kurepa \cite{kurepa1956ecart} and Todorcevic \cite{todorvcevic1988lindelof},  \cite[Theorem~9.14]{todorvcevic1984trees}. 
By \emph{tree} we mean a partially ordered set $T$ with a unique minimal element such that the set of predecessors of any element of $T$ is well-ordered. A tree is \emph{special} if it is a union of countably many antichains. Our trees may have distinct limits nodes with the same set of strict predecessors. 
A \emph{path} in $T$ is a down-closed chain. A path without maximal element is a \emph{ray}, and an inclusion-wise maximal path is a \emph{branch}. By identifying paths with their characteristic function, the set $\cP(T)$ of all paths of $T$ is a closed, hence compact subspace of $2^{T}$. The space $\cP(T)$ is the \emph{path space} associated with $T$. The \emph{ray space} $\cR(T)$ and the \emph{branch space} $\cB(T)$ are its  subspaces consisting of all rays respectively branches of $T$. 
 Identifying nodes $t$ of $T$ with the path with maximum $t$ gives a dense embedding $T \subseteq \cP(T)$, and so we may view $\cP(T)$ as a  compactification of $T$. In this perspective, the boundary $\cP(T) \setminus T$ of this compactification is the ray space $\cR(T)$ of $T$. %Note that for any graph-theoretic tree we have $\cR(T) = \Omega(T)$.
 
 Path- and ray spaces of special trees are of particular importance. Gruenhage showed in \cite{Gruenhage1986} that a path space $\cP(T)$ is Eberlein compact if and only if $T$ is special. The significance of ray spaces of special trees %for the study of end spaces of infinite graphs 
 lies in the following result, see also Section~\ref{sec_repthm} below for a more detailed discussion:
 
 \begin{thm}[{Kurkofka \& Pitz \cite{kurkofkapitz_rep}}]
 \label{thm_kp_rep}
	The following are equivalent for a topological space $X$:
	\begin{enumerate}
		\item $X$ is homeomorphic to the end space of a graph,
		\item $X$ is homeomorphic to the ray space of a special tree.
 	\end{enumerate}
\end{thm}
Hence, to understand end spaces of infinite graphs, it suffices to understand topological properties of ray spaces of (special) trees, and this is the programme carried out in this paper.

It turns out that ray spaces of trees are far easier to analyse than end spaces of graphs, the main reason being that ray spaces have a canonical (sub-)base. 
As a case in point, recall that in the framework of end spaces, establishing the existence of a \emph{nested} family of clopen sets witnessing that end spaces are Hausdorff is a deep result by Carmesin \cite[Corollary~5.17]{carmesin2019all}. See \cite{diestel2018tree, gollin2021representations} for further applications of nested sets of separations of finite and infinite graphs. In the framework of ray spaces, the definition $\cR(T) \subseteq 2^T$ makes it clear that the topology of $\cR(T)$ is generated by the subbase obtained by declaring all sets of the form 
$[t] = \set{x \in \cR(T)}:{t \in x}$
 to be clopen. But now it is obvious that this collection $\cS = \set{[t]}:{t \in T}$ is indeed \emph{nested} in the sense that whenever two elements from $\cS$ intersect, one of them is included in the other.

As one set of results, we use this canonical subbase to establish the following topological properties of ray spaces, which were previously hard or unknown for end spaces:\footnote{We refer the reader to Sections~\ref{subsec_baseproperties} and \ref{sec_baseII} for any undefined topological terms.}
 
 \begin{enumerate}[label=(\roman*)]
 		\item \emph{Every ray space, and hence every end space, is ultraparacompact} (Proposition~\ref{prop_ultraparacompactness}). 
 			That end spaces of infinite graphs are ultraparacompact has been proven only recently, using Jung's theory of normal trees \cite{kurkofka2021approximating}. 
 		\item \emph{Every ray space, and hence every end space, satisfies all Amsterdam completeness properties, and hence is Baire} (Proposition~\ref{prop_basecpt2}). 
 			This result may be viewed as a considerable strengthening of the Diestel-Kühn \emph{direction theorem} for ends in infinite graphs \cite[Theorem~2.2]{diestel2003graph}.		%The direction theorem says that for the standard base of $\Omega(G)$, any centered collection $\Set{C(X,\omega_X)}$ where $X$ runs through \emph{all} finite subsets of $V(G)$ has non-empty intersection. The point is that this filter is maximal in the sense that for each finite set $X$, one component $C(X,\omega_X)$ has been chosen. However, it is not the case that any partial direction can be extended to a direction, so the standard base for end spaces of graphs is not base-complete: Consider an infinite star of rays, and consider on each ray a separation which points towards the root. This collection is centered, but has empty intersection.
	\item \emph{Every ray space, and hence every end space, is monotonically normal} (Proposition~\ref{prop_monnormal}).
	\item \emph{Every ray space of a special tree, and hence every end space, contains a dense, completely metrizable subspace} (Proposition~\ref{prop_densemetr}).
	\item \emph{Every compact ray space of a special tree, and hence every compact end space, is Eberlein compact} (Corollary~\ref{cor_endspaceEberlein}). 
 \end{enumerate}

These properties should be seen in the light of the result that the metrizable end spaces are precisely the completely ultrametrizable spaces (see Theorem~\ref{thm_charmetr} below), and so it is natural to inquire to what degree the latter properties lift to general, non-metrizable end spaces.  

Our second set of results gives topological characterisations of (subspaces of) path-, branch- and ray spaces of (special) trees. 
 Since the class of path spaces corresponds precisely to the class of compact ray spaces (Theorem~\ref{thm_charpath}), it suffices to focus on ray- and branch spaces. Here, our main result reads as follows:

\begin{thm}
\label{thm_charray}
	The following are equivalent for a Hausdorff space $X$:
	\begin{enumerate}
		\item $X$ is homeomorphic to the ray space of a \textnormal{[}special\textnormal{]} tree, 
		\item $X$ admits a nested clopen subbase that is
 noetherian, hereditarily complete \textnormal{[}and $\sigma$-disjoint\textnormal{]}.
 	\end{enumerate}
\end{thm}

The crucial new ingredient is to identify the correct completeness condition: Indeed, branch- and ray spaces are generally not \v{C}ech complete (Examples~\ref{example_onepointcptres} and \ref{example_Michaelline} below); instead, we use a new completeness property (the existence of a `hereditarily complete' subbase) inspired by the Amsterdam completeness properties from \cite{completeness}.

Observing that all characterising conditions in Theorem~\ref{thm_charray}(2) are closed-hereditary, we conclude that the property of being a ray space is closed hereditary.

By omitting the completeness condition in Theorem~\ref{thm_charray}(2), we obtain precisely the subspaces of ray spaces (see Theorem~\ref{thm_charrayrepeat} for details), and by replacing the word `subbase' by `base' in the above theorems, one obtains corresponding characterisations of branch spaces (Theorem~\ref{thm_charbranch}), answering a question by P.~Nyikos \cite[p.8]{nyikos1999some}. 

Theorems~\ref{thm_kp_rep} and \ref{thm_charray} combined give the topological characterisation of the class of (subspaces of) end spaces of infinite graphs announced in the abstract.

\medskip
\textbf{Acknowledgements.} I would like to thank K.P.~Hart and Stevo Todorcevic for bringing the path space topology to my attention, and Will Brian for a suggestion leading to the present  characterisations of ray spaces in terms of (sub-)bases.

\section{Topological properties of ray spaces}
\label{sec_2}

\subsection{Order trees and ray spaces}

Let $T$ be an order tree. Paths, rays and branches of trees have been defined in the introduction. Every branch inherits a well-ordering from~$T$. The
\emph{height} of~$T$ is the supremum of the order types of its branches. A \emph{graph-theoretic tree} is a tree of height at most $\omega$. 

Write $\lceil t \rceil = \set{s \in T}:{s \leq t}$ for the set of predecessors (plural) of a node $t \in T$. The
\emph{height} of a point $t\in T$ is the order type of the set of strict predecessors~$\mathring{\lceil t \rceil}  :=
\lceil t \rceil  \setminus \{t\}$. For an ordinal $i$, the set $T^i$ of all points at height $i$ is
the $i$th \emph{level} of~$T$, and we
write $T^{<i} := \bigcup\set{T^j}:{j < i}$.

The intuitive interpretation of a tree order as expressing height will also
be used informally. For example, we say that $s$ is \emph{above}~$t$
if $t \leq s$, and write $\lfloor t \rfloor = \set{s \in T}:{t \leq s}$ for all nodes above $t$.  If $t < t'$, we write $[t,t'] = \set{x}:{t \leq x \leq t'}$, and call this
set a (closed) \emph{interval} in~$T$. (Open and half-open intervals in~$T$
are defined analogously.) Any~subset of $T$ is an order tree under the
ordering induced by~$T$. A subset $T' \subseteq T$ is a \emph{subtree} of $T$ if along with any two
comparable points it contains the interval in~$T$ between them, and it is a \emph{rooted} subtree if it additionally contains the root of $T$. Thus, a path is a non-splitting rooted subtree of $T$. If $t < t'$
but there is no point between $t$ and~$t'$, we call $t'$ a \emph{successor}
of~$t$ and $t$ the \emph{predecessor} (singular) of~$t'$. We write $\operatorname{succ}(t)$ for the set of successors of a node $t \in T$. A node without a successor is a \emph{leaf}. A tree without leaves is \emph{pruned}. If $t$ is not a successor
of any point it is called a \emph{limit}.

A \emph{top} of a ray $x$ in an order tree $T$ is a node $t \in T$ with $\mathring{\dc{t}}=x$.
Note that for Theorem~\ref{thm_kp_rep} to hold, we explicitly allow that a ray $x \subseteq T$ may have multiple tops, even infinitely or uncountably many, as in Examples \ref{example_onepointcpt} and \ref{example_onepointcptres} below.

\subsection{Ray spaces and the standard (sub-)base}

As mentioned in the introduction, the topology of the path space $\cP(T) \subseteq 2^T$ is generated by all open sets of the form $[t] = \set{x \in \cP(T)}:{t \in x}$ and their complements $[t]^\complement = \set{x \in \cP(T)}:{t \notin x}$.
In fact, note that the collection $\set{[t]}:{t \in T}$ induces a base for the branch space $\cB(T)$. This is generally not true for the ray space $\cR(T)$: 

\begin{lemma}
\label{lem_standardbase}
	For ray spaces, a local neighbourhood base $\cC(x)$ at some ray $x \in \cR(T)$ is given by sets of the form
$$[t,F] := [t] \setminus \bigcup_{s \in F} [s] = [t] \setminus [F] $$
where $t \in x$ and $F$ is a finite set of tops of $x$. 
\end{lemma}
\begin{proof}
	By definition of our subbase, every open set $U$ with $x \in U \subseteq \cR(T)$ contains a basic set $B$ of the form
$$x \in B = \bigcap_{t \in E} [t] \cap \bigcap_{s \in F} [s]^\complement= \bigcap_{t \in E} [t] \setminus \bigcup_{s \in F} [s] \subseteq U$$
for some finite sets $E,F \subset T$. Pick $B$ with $|E| + |F|$ of minimal size. Since $x \in B$ implies $E \subset x$, we could replace $E$ by its maximum. Hence, $|E| \leq 1$. % and $$x \in [t] \setminus \bigcup_{s \in F} [s] \subseteq U.$$
Next, we claim that for every $s\in F$ there is a top $s'$ of  $x$ with $s' \leq s$. Indeed, if say $\dc{s_0} \cap x \subsetneq x$, pick $t \in x \setminus (E \cup \dc{s_0})$ to obtain
$$x \in [t] \setminus \bigcup_{s \in F_0} [s] \subseteq B$$
for $F_0 = F \setminus \Set{s_0}$, contracting the minimality of $|E|+2 |F|$. Hence, for every $s\in F$ there is a top $s'$ of  $x$ with $s' \leq s$, and by setting $F' = \set{s'}:{s \in F}$, we obtain
\[x \in [t,F'] \subseteq B \subseteq U \subseteq \cR(T) \]
with $t \in x$ and $F'$ a finite set of tops of $x$ as desired.
\end{proof}

  The \emph{standard base} of $\cR(T)$ is the collection 
$$\cC = \bigcup_{x \in \cR(T)} \cC(x) =  \set{ [ t,F]}:{ t \in T, \; F \subset T \text{ finite set of limits}, \; t \in \mathring{\dc{s}} = \mathring{\dc{s'}} \text{ for all } s \neq s' \in F},$$
%is the \emph{standard base} for ray spaces $\cR(T)$, 
and elements of $\cC$ will be called \emph{standard basic opens}. In Section~\ref{subsec_baseproperties} we will investigate further properties of the standard base. But before that, we look at some typical examples of ray spaces.

\subsection{Examples of ray spaces}

\begin{exmp}
The \emph{Cantor set} can be represented as the ray space of a full binary tree $2^{<\omega}$.
More generally, for pruned trees $T$ of height $\omega$ we have $\cR(T) = \cB(T)$, and these are precisely the completely ultrametrizable spaces, see e.g.~\cite{hughes2004trees}.
\end{exmp}

\begin{exmp}
\label{example_onepointcpt} The \emph{one-point compactification of an uncountable discrete space of size $\kappa$} can be represented as the ray space of the pruned tree $T$ of height $\omega \cdot 2$ such that every node has precisely one successor, but $T^\omega$, the $\omega$th level of $T$, has size $\kappa$. 
\end{exmp}

\begin{proof}
	Indeed, since $|T^\omega| = \kappa$ there are $\kappa$ branches, and each branch is isolated in $\cR(T)$. Now consider the path $x = T^{<\omega}$. By the definition of the standard base, a local base at $x$ is given by the sets
	$$[t,F] := [t] \setminus \bigcup_{s \in F} [s] = [t] \setminus [F] $$
where $t \in x$ and $F \subset T^\omega$ is a finite set of tops of $x$. Hence, $[t,F]$ contains all but finitely many branches, and so $\cR(T)$ is homeomorphic to the one-point compactification of an uncountable discrete space of size $\kappa$.
\end{proof}

\begin{exmp}
\label{example_onepointcptres} A \emph{resolution of the one-point compactification of an uncountable discrete space of size $\kappa$ in which every isolated point has been replaced by a clopen copy of $\NN$}. This space can be represented as the ray space of the pruned tree $T$ of height $\omega \cdot 2$ where every non-limit node has precisely one successor, every limit node has countably many successors, and $T^\omega$, the $\omega$th level of $T$, has size $\kappa$.
\end{exmp}

%\textcolor{blue}{cite Watson}

%\begin{proof}
%	xx
%\end{proof}

\begin{lemma}
	The ray space in Example~\ref{example_onepointcptres} fails to be \v{C}ech-complete. In particular, end spaces of graphs are generally not \v{C}ech-complete.
\end{lemma}

\begin{proof}
	Using the internal characterisation of \v{C}ech-completeness as in \cite[Theorem~3.9.2]{engelking1989book}, it suffices to show: Given any sequence of open covers $\cU_n$ ($n \in \NN$) of $\cR(T)$, there is a filter $\cG$ of closed sets such that $\cG$ is \emph{less than} $\cU_n$ for all $n \in \NN$ (meaning there are $G_n \in \cG$ and $U_n \in \cU_n$ with $G_n \subseteq U_n$), but $\bigcap \cG = \emptyset$. 
	
	Write $x = T^{<\omega} \in \cR(T)$ for the unique non-isolated ray in $T$. For every $n \in \NN$ fix $U_n \in \cU_n$ with $x \in U_n$, and choose a standard basic open $x \in [t_n,F_n] \subseteq U_n$ where $F_n$ is a finite set of tops of $x$. Then $\bigcup F_n$ is countable. As $|T^\omega|$ is uncountable, there is a top $s$ of $x$ not contained in any $F_n$. Note that $[s] \cong \NN$ is a closed copy of the countable discrete space contained in $\bigcap_{n \in \NN} U_n$. But then the filter $\cG$ generated by the cofinite filter on $[s]$ is less than $\cU_n$ for all $n \in \NN$, but $\bigcap \cG = \emptyset$. %Hence, $\cR(T)$ fails to be \v{C}ech-complete.
	\end{proof}

\begin{exmp}
The \emph{Alexandroff duplicate of a Cantor set} can be represented as the ray space of the pruned tree $T \subset 2^{<\omega \cdot 2}$ such that $2^{\leq \omega} \subset T$ and every node of height $\geq \omega$ has precisely one successor.
\end{exmp}

\begin{proof}
	The Alexandroff duplicate $\mathcal{A}(X)$ of space $X$ is the space on the set $\Set{0,1} \times X$ where every point in $\Set{1} \times X$ is isolated any every point $(0,x)$ has a neighbourhood base of the form $(\Set{0,1} \times U)  \setminus \Set{(1,x)}$ for an open neighbourhood $x \in U \subseteq X$ of $x$ in $X$.
	To see that the ray $\cR(T)$ space in question is homeomorphic to the Alexandroff duplicate of a Cantor set, we verify that the map 
	$$f \colon \mathcal{A}(\cR(2^{<\omega})) \to \cR(T), \; (i,x) \mapsto \begin{cases}
		x & \text{if } i=0, \\ b \in \cB(T), \; x \subsetneq b & \text{if } i=1
	\end{cases}$$
	is a homeomorphism. It is clearly bijective; by compactness, %of the Alexandroff duplicate, 
	it remains to show that $f$ is continuous, i.e.\ that preimages of subbasic clopen sets are open. 	So let $[t] \subset \cR(T)$ be any subbasic clopen set. If $t \notin T^{<\omega}$, then $[t]$ consists of an isolated branch $b$, and $f^{-1}(b) = \Set{(1,x)}$ is isolated, too. Otherwise, if $t \in T^{<\omega}$, then $t$ determines a basic clopen set $U$ of the Cantor set, and $f^{-1}([t]) = \Set{0,1} \times U$ is clopen.
	 \end{proof}
	 
	 With a similar construction, we see that  every Alexandroff duplicate of a competely ultrametrizable space can be represented by a ray space. 
	 Call a ray of the binary tree $2^{<\omega}$ \emph{rational} if its corresponding $0-1$-sequence becomes eventually constant, and \emph{irrational} otherwise. 
	 
\begin{exmp}
\label{example_Michaelline} The \emph{Michael-line} can be represented as the branch space of the pruned subtree $T'$ of the tree $T$ from the previous example, in which we delete all elements of $T$ above a rational ray of $2^{<\omega}$. 
\end{exmp}

\begin{lemma}
\label{lem_Michaelline}
	The Michael-line branch space is Baire, but contains a closed subspace which is not Baire, namely a closed copy of $\mathbb{Q}$. 
	\end{lemma} 
	\begin{proof}
	We show below in %Corollary~\ref{cor_bairebranch}
	Proposition~\ref{prop_basecptbranch} that branch spaces are Baire. However, the set of all branches of order type $\omega$ in $\cB(T')$ forms a crowded, countable metrizable space, so a copy of $\mathbb{Q}$ by Sierpinski's characterisation \cite{sierpinski1920propriete}. Since all remaining branches in $\cB(T')$ are isolated, this copy of $\mathbb{Q}$ is closed.
	\end{proof}

%The following completeness properties are known as the \emph{Amsterdam properties}. \cite{completeness} A space $X$ is
% \begin{itemize}
% 	\item \emph{cocompact} if there is a base $\cB$ of open sets for $X$ such that whenever $\cC \subset \cB$ and $\closure{\cC}:=\set{\closure{C}}:{C \in \cC} $ is centered, then $\bigcap \closure{\cC} \neq \emptyset$. 
%	\item \emph{base-compact} if there is a base $\cB$ of open sets for $X$ such that whenever $\cC \subset \cB$ is centered, then $\bigcap \cC \neq \emptyset$; and
%	\item \emph{subcompact} if there is a base $\cB$ of open sets for $X$ such that whenever $\cC \subset \cB$ is a regular filter base, then $\bigcap \cC \neq \emptyset$.
% \end{itemize} 

\subsection{The representation theorem}
\label{sec_repthm}

Our Representation Theorem~\ref{thm_kp_rep} has been phrased somewhat differently than the original version in \cite{kurkofkapitz_rep}. To see that both versions are equivalent, recall that the original topology $\tau_{seq}$ on $\cR(T)$ for a special tree $T$ has been described in terms of convergent sequences \cite[Lemma~5.3]{kurkofkapitz_rep} as follows:
Let $x$ and $x_n$ ($n \in \NN$) be rays in a special tree $T$. Let $A\subset\NN$ consist of all numbers~$n$ for which $x \subsetneq x_n$, and let $B:=\NN\setminus A$.
\begin{enumerate}
    \item We have convergence $x_n \to x$  for $n \in A$ and $n\to\infty$ if and only if $A$ is infinite and for every top $t$ of $x$ there are only finitely many $n \in A$ with $t \in x_n$.
    \item We have convergence $x_n \to x$  for $n \in B$ and $n\to\infty$ if and only if $B$ is infinite and for every node $t \in x$ there are only finitely many $n \in B$ with $x \cap x_n \subset \mathring{\dc{t}}$.
\end{enumerate}
Write $\tau$ for the topology on $\cR(T) \subseteq 2^T$ as defined in the present paper. In order to verify that $\tau$ and $\tau_{seq}$ induce the same closed sets, given $X \subseteq \cR(T)$ we need to show that $x \in \closure{X}$ with respect to $\tau$ if and only if there is a sequence $(x_n) \subseteq X$ such that $x_n \to x$ in $\tau_{seq}$. 

For the forwards implication, suppose that $x \in \closure{X} \setminus X$. If there are infinitely many tops $s_n$ ($n \in \NN$) of $x$ such that $[s_n] \cap X \neq \emptyset$ then select $x_n \in [s_n] \in X$ and note that $x_n \to x$ in $\tau_{seq}$ according to (1). Hence, we may suppose that there are only finitely many tops $F = \Set{s_1,\ldots, s_k}$ of $x$ such that $X \cap [s_i] \neq \emptyset$. Let $X' = X \setminus [F]$. Since $[F] \not\ni x$ is clopen, it follows that $x \in \closure{X'}$. Let $t_n$ be an increasing, cofinal sequence in $x$ (which exists since $T$ is special). Since $\bigcap_{n \in \NN}[t_n] \cap X' = \emptyset$ but $[t_n] \cap X' \neq \emptyset$, we can choose pairwise distinct $x_n \in X' \cap [t_n]$, and we have convergence $x_n \to x$ in $\tau_{seq}$ according to (2).

For the backwards implication, suppose $x_n \to x$ in $\tau_{seq}$. We show that $x \in \closure{X}$ with respect to $\tau$. Let $[t,F]$ be a standard basic open neighbourhood of $x$ where $t \in x$ and $F$ a finite set of tops of $x$. If $x_n \to x$ according to (1), then only finitely many members of $\set{x_n}:{n \in \NN}$ lie above $F$, so infinitely many $x_n$ belong to $[t,F]$.  If $x_n \to x$ according to (2), then only finitely many members of $x_n$ do not contain $t$, so infinitely many $x_n$ belong to $[t,F]$.

\subsection{Base properties I}
\label{subsec_baseproperties}
\label{sec25}

Let $\cA$ be any collection of sets, for example a base, a subbase or a cover of some topological space.

	\begin{itemize}
	\item $\cA$ is \emph{nested} if for all $A,B \in \cA$ with $A \cap B \neq \emptyset$ we have $A \subseteq B$ or $B \subseteq A$. In the topological literature, nested set collections are also called \emph{rank-1} collections or \emph{nonarchimedean} collections \cite{nyikos1975some,nyikos1999some}, whereas in combinatorics and computer science, also the term \emph{laminar} collection is used  \cite{cheriyan19992}.
	\item $\cA$ is \emph{noetherian} if every $\subseteq
$-increasing sequence of elements from $\cA$ is eventually constant.
	\item $\cA$ is \emph{centered} if any finite subfamily of $\cA$ has non-empty intersection.
	\item $\cA$ is \emph{complete} if any centered subfamily of $\cA$ has non-empty intersection. We remark that the existence of a complete base of clopen sets for $X$ implies any of the three Amsterdam completeness properties \emph{co-compactness, base-compactness} and \emph{sub-compactness}, see \cite{completeness}, which in turn all imply the Baire property.
	\item $\cA$ is \emph{hereditarily complete} if for every closed subspace $Y \subseteq X$, the collection $\cA \restriction Y = \set{A \cap Y}:{A \in \cA}$ is complete.
	\item We say a space is \emph{base-complete} is it admits a complete base.
\end{itemize}

In the remainder of this section, we investigate properties of the standard base $\cC$ of a ray space $\cR(T)$, and show, using standard basic open sets, that ray spaces are  monotonically normal.
\begin{lemma}
\label{lem_increasing2}
For any ray $x \in \cR(T)$, its local neighbourhood base $\cC(x)$ is noetherian.
\end{lemma}
\begin{proof}
We show that every increasing sequence of standard basic open sets in $\cC(x)$ is eventually constant. Suppose for a contradiction that
$$[t_0,F_0] \subsetneq [t_1,F_1] \subsetneq [t_2,F_2] \subseteq \cdots$$
is an infinite, strictly increasing sequence of standard basic open sets. There are two ways to make a basic open set $[t,F] \in \cC(x)$ strictly bigger: either decrease $t$ in $(T,\leq)$ or take a strict subset of $F$ (where $F$ is a finite set of tops of $x$). The latter option one can iterate only finitely often, so there is an infinite subsequence $n_i$ with 
$$t_{n_0} > t_{n_1} > t_{n_2} > \cdots.$$
But this contradicts well-foundedness of $T$.
\end{proof}

\begin{prop}
\label{prop_increasing}
In any ray space $\cR(T)$, the collection $\cS=\set{[t]}:{t \in T}$ is noetherian. \hfill \qed
\end{prop}

\begin{prop}
\label{prop_basecptbranch}
Branch spaces are base-complete, and hence Baire. 
\end{prop}

\begin{proof}
Consider any branch space $\cB(T)$. We show that the standard base $\cS=\set{[t]}:{t \in T}$ is complete.

	Let $\cA \subseteq \cS$ be any centered subcollection. Every $A \in \cA$ is of the form $A=[t_A]$. Since $\cA$ is centered, the set $M=\set{t_A}:{A \in \cA}$ lies on a chain of $T$. Let $x$ be any branch of $T$ with $M \subseteq x$. Then $x \in \bigcap \cA \neq \emptyset$.
\end{proof}

\begin{prop}
\label{prop_basecpt}
Ray spaces are base-complete, and hence Baire.%\footnote{This completeness result is new already for end spaces of infinite graphs. The usual base $\Omega(S,\omega)$ (see \cite[\S8.6]{diestel2015book}) for the end space $\Omega(G)$ of an infinite graph $G$ is not complete in this sense: Consider an infinite star of rays, and take on each ray a separation which points towards the root. This collection of basic opens is centered, but has empty intersection. The previously best known completeness-type result for end spaces of graphs has been the Diestel-Kühn \emph{direction theorem} from \cite[Theorem~2.2]{diestel2003graph}, which says that any \emph{maximal} centered collection $\set{\Omega(S,\omega_S)}:{S \subset V(G) \text{ finite}}$, i.e.\ a choice of basic opens where $S$ runs through \emph{all} finite subsets of $V(G)$, has non-empty intersection.}
\end{prop}

\begin{proof}
Consider a ray space $\cR(T)$. We may assume that $T$ is pruned. We show that the standard base $\cC$ from Lemma~\ref{lem_standardbase} is complete.
%Let $\cB$ be the standard base for $\cR(T)$, and 

Let $\cA \subseteq \cC$ be any centered subcollection. Every $A \in \cA$ is of the form $A=[t_A,F_A]$. Since $\cA$ is centered, the set $M=\set{t_A}:{A \in \cA}$ lies on a chain of $T$. 

Let $x$ be an inclusion minimal ray of $T$ with $M \subseteq x$. We claim that $x \in \bigcap \cA$. Otherwise, there is $A=[t_A,F_A] \in \cA$ with $x \notin [t_A,F_A]$. Since $t_A \in M \subseteq x$, there must exist $\ell \in F_A \cap x$. By definition of the standard base, $\ell$ is a limit, so $\mathring{\dc{\ell}}$ is a ray. But $\mathring{\dc{\ell}} \subsetneq x$, so by minimality of $x$ there is $t_{A'} \in M$ with $t_{A'} \geq \ell$. But then $A \cap A' = \emptyset$, contradicting that $\cA$ is centered.
\end{proof}

\begin{prop}
\label{prop_basecpt2}
In any ray space $\cR(T)$, the collection $\cS=\set{[t]}:{t \in T}$ is hereditarily complete.\end{prop}

\begin{proof}
Let $Y \subseteq \cR(T)$ be a closed subspace, and let $\cA \subseteq \cS$ be any subcollection such that  $\cA \restriction Y$ is centered.
	Using Zorn's lemma, 	we may assume that $\cA$ is a maximal such collection. 		
	Every $A \in \cA$ is of the form $A=[t_A]$. Since $\cA$ is centered, the set $M=\set{t_A}:{A \in \cA}$ lies on a chain of $T$, and by maximality of $\cA$, the set $M$ is a path.
	
If the path $M$ has a maximal element $m$, then trivially $\bigcap \cA \cap Y = [m] \cap Y \neq \emptyset$. 

If the path $M$ has no maximal element, then $x = M \in \cR(T)$ is a ray. Clearly, $x \in \bigcap \cA$. It remains to show that $ x \in \overline{Y} = Y$. Towards this end, let $[t,F]$ be a standard basic open neighbourhood of $x$ in $\cR(T)$ (i.e.\ $t \in x$ and all $s \in F$ are tops of $x$ in $T$). We claim that $[t,F] \cap Y \neq \emptyset$.

If $[s] \cap Y \neq \emptyset$ for some $s \in F$, then we could add $[s]$ to $\cA$ without violating that $\cA \restriction Y$ is centered, contradicting the maximality of $\cA$. Hence, $[F] \cap Y = \emptyset$. Since $t \in x = M$, we know $[t] \in \cA$, so $[t] \cap Y \neq \emptyset$. Together with that fact that $[F] \cap Y = \emptyset$ it follows that $[t,F] \cap Y \neq \emptyset$, as desired.
\end{proof}

A topological space $X$ is \emph{monotonically normal} if to each pair $\pair{U, x}$ where $U$ is an open set and $x \in U$, it is possible to assign an open set $U_x$ with $x \in U_x \subseteq U$ such that $U_x \cap V_y \neq \emptyset $ implies $x \in V$ or $y \in U$.

\begin{prop}
\label{prop_monnormal}
Ray spaces are monotonically normal, and hence hereditarily normal.
\end{prop}

\begin{proof}
Given any ray $x \in \cR(T)$ and open $U \ni x$, simply let $U_x := [t_x,F_x]$ be a standard basic open neighbourhood with $x \in [t_x,F_x] \subseteq U$. We show that these $U_x$ are as required.

So consider pairs $x \in U$ and $y \in V$ with $x \notin V$ and $y \notin U$.

Case 1: The rays $x$ and $y$ are comparable, say $x \subsetneq y$. Since $y \notin U \supseteq [t_x,F_x]$, there must be $s \in F_x$ with $s \in y$. Moreover, since $x \notin V \supseteq [t_y,F_y]$, we must have $x \subseteq \dc{t_y}$. Therefore, $s \leq t_y$, and so $[t_y] \subseteq [F_x]$, implying $[t_x,F_x] \cap [t_y,F_y] = \emptyset$. 

Case 2: The rays $x$ and $y$ are incomparable. Since $y \notin U \supseteq [t_x,F_x]$ and $y \notin [F_x]$, we must have $y \notin [t_x]$. Similarly, $x \notin [t_y]$. Hence, $t_x$ and $t_y$ are incomparable, so $[t_x,F_x] \cap [t_y,F_y]  \subseteq [t_x] \cap [t_y] = \emptyset$.
\end{proof}

\subsection{Compactness properties}
\label{sec26}

It is well-known and easy to show that all branch spaces are hereditarily ultraparacompact \cite{nyikos1999some}. Much more difficult is that also all end spaces hereditarily ultraparacompact \cite{kurkofka2021approximating}, and hence also all ray spaces of special trees by Theorem~\ref{thm_kp_rep}.
If $T$ has an uncountable chain, then we get a copy of $\omega_1+1$ in $\cR(T)$, so are no longer hereditarily paracompact, as witnessed by the subspace $\omega_1$. But we can still show that all ray spaces are ultraparacompact, irrespectively of whether they contain uncountable branches or not.

\begin{prop}
\label{prop_ultraparacompactness}
Ray spaces are ultraparacompact: Every open cover can be refined to an open partition consisting of standard basic open sets.
\end{prop}

\begin{proof}
Let $\cC$ be the standard base for some ray space $\cR(T)$, and let $\cU$ be an open cover of $\cR(T)$. For ultraparacompactness, we need to find an open partition $\cB$ refining $\cU$. 
By transfinite recursion of length at most $\tau = |T|^+$, we will define nested collections $\cB_i \subset \cC$ (for $i < \tau$) %of standard basic open sets of $\cR(T)$ 
such that
\begin{enumerate}
	\item $\cB_i \subseteq \cB_j$ whenever $i \leq j < \tau$,
    \item $\cB_i$ refines $\cU$ for all $i < \tau$,
    \item $\cB_i $ is up-closed in the sense if $x, x' \in \cR(T)$ with $x \subseteq x'$ and $x \in \bigcup \cB_i$, then $x' \in \bigcup \cB_i$, too.
\end{enumerate}
We let $\cB_0 = \emptyset$. If the collection $\cB_i$ refining $\cU$ is already defined and it does not yet cover $\cR(T)$, let $x_i \in \cR(T) \setminus \bigcup \cB_i$ be any inclusionwise maximal ray with this property (i.e.\ every ray properly extending $x_i$ is covered by $\cB_i$). This is possible as $\bigcup \cB_i$ is open. Choose $U_i \in \cU$ with $x_i \in U$ and let $[t_i,F_i]$ be any standard basic open neighbourhood of $x \in [t_i,F_i] \subset U$ in $\cR(T)$ (i.e.\ where $F_i$ is a set of tops of $x$), and put $\cB_{i+1} = \cB_i \cup \Set{[t_i,F_i]}$. Clearly, $\cB_{i+1}$ refines $\cU$.
To see that $\cB_{i+1}$ is still nested, consider any element $[t_j,F_j] \in \cB_i$ such that $[t_j,F_j] \cap [t_i,F_i] \neq \emptyset$. %, and fix $x \in [t,F] \cap [t_i,F_i]$. 
Then $t_i$ and $t_j$ are comparable. By property (1) and the fact that we deal with standard basic opens, we must have $t_i \leq t_j$ and $t_j \notin x$. In particular, whenever $[t_i,F_i]$ intersects a previously chosen $[t_j,F_j]$, then 
%$$t_i < t. \quad\quad \quad (\star)$$
\[
t_i < t_j \quad \text{for all } j < i. \label{eq:gettingsmaller} \tag{$\star$}
\]
 Now either $t_j \in \uc{F_i}$, but then $[t_j,F_j] \cap [t_i,F_i] =\emptyset$, or otherwise $[t_j,F_j] \subseteq [t_i,F_i]$.
At limit ordinals $\ell$, we define $\cB_{\ell} = \bigcup \set{\cB_i}:{i < \ell}$, which will continue to be nested by (1) and the fact that nestedness is a finitary property. 
This definition also satisfies (2) and (3). This completes the recursive construction.

Since all $t_i$ are distinct, this recursion terminates after at most $\tau = |T|^+$ steps. Let $\cB_{\tau} = \bigcup \set{\cB_i}:{i < \tau}$ be the final nested collection. Property $(\star)$ implies that every $\subseteq$-increasing chain in $\cB_\tau$ is finite. Let $\cB \subseteq \cB_\tau$ be the collection of $\subseteq$ maximal elements. Since $\cB_\tau$ was a cover of $\cR(T)$ and every element of $\cB_\tau$ is contained in a maximal element, also $\cB$ is a cover. And since $\cB_\tau$ was nested, it is clear that any two maximal elements of $\cB_\tau$ are disjoint or equal. Thus, $\cB$ is an open partition of $\cR(T)$. Since $\cB \subseteq \cB_\tau$ and $ \cB_\tau$ refines $\cU$, it follows that $\cR(T)$ is ultraparacompact.
\end{proof}

The \emph{Lindel\"of number} of a space is the smallest cardinal $\kappa$ such that every open cover of the space $X$
 has a subcover of size $< \kappa$. So a space has Lindel\"of number $\aleph_0$ if and only if it compact, and has Lindel\"of number $\aleph_1$ if and only if it Lindel\"of in the usual sense (every open cover has a countable subcover).

\begin{prop}
Let $\kappa$ be a regular cardinal.
A ray space $\cR(T)$ of a pruned tree $T$ has Lindel\"of number $\kappa$ if and only if every node of $T$ has $<\kappa$ successors. 
\end{prop}

Observe that no conditions on the number of tops of a ray in $T$ are required, cf.~Example~\ref{example_onepointcpt}.

\begin{proof}
If some node $t$ has at least $\kappa$ many successors, then 
$$\big\{[t]^\complement\big\} \cup \big\{[s] \colon s \text{ a successor of } t \big\}$$ is an open partition of $\cR(T)$ without $<\kappa$ sized subcover.

Now assume that conversely, every node of $T$ has $<\kappa$ many successors. Suppose further for a contradiction that there is an open cover $\cU$ without $<\kappa$ sized subcover. By Proposition~\ref{prop_ultraparacompactness} we may assume that $\cU = \set{[t_i,F_i]}:{i \in I}$
is an open partition consisting of non-empty standard basic opens, for some index set $I$ with $|I| \geq \kappa$. Note that disjointness implies that $t_i \neq t_j$ for $i \neq j \in I$. Let $M = \set{t_i}:{i \in I}$.

Let $S = \set{t \in T}:{|\uc{t} \cap M| \geq \kappa } \subseteq T$. Clearly, $S$ is down-closed, and non-empty, and by the condition that every node in $T$ has $<\kappa$ successors and $\kappa$ is regular, we get that $S$ is a rooted, pruned subtree of $T$.

Now let $x$ be a branch of $S$; since $S$ is pruned, $x \in \cR(T)$. Let $[t_*,F_*]$ be the element of $\cU$ containing $x$. 
Since $t^* \in x \subseteq S$, there exist $\kappa$ elements from $M$ above $t^*$. By disjointness of $\script{U}$, all of these $\kappa$ elements belong to $\uc{F^*}$. But since $F^*$ is finite, there are $\kappa$ many elements above some $s \in F^*$, so $s \in S$. But then $x \subsetneq \dc{s} \subseteq S$, contradicting the maximality of $x$.
\end{proof}

\begin{cor}
A ray space $\cR(T)$ of a pruned tree $T$ is compact if and only if every node of $T$ has only finitely many successors
\hfill \qed
\end{cor}

\begin{cor}
A ray space $\cR(T)$ of an uncountable pruned tree $T$ such that every node has only countably many successors is never metrizable.
\end{cor}
\begin{proof}
Lindel\"of plus metrizable implies second countable, but the ray space of an uncountable pruned tree $T$ has uncountable weight: To see this, first observe that if $T$ has an uncountable level $T^\alpha$, then $\set{[t]}:{t \in T^\alpha}$ witnesses that $\cR(T)$ has uncountable cellularity, so uncountable weight.

Otherwise, all levels are countable, so $T$ has uncountable height. Consider $T' = T^{<\omega_1}$. If $\cR(T')$ has countable weight, there would be a countable basis $\cC$  of $\cR(T')$ consisting of standard basic open sets $\set{[t_n,F_n]}:{n \in \NN}$. Let $M = \set{t_n}:{n \in \NN} \cup \bigcup \set{F_n}:{n \in \NN}$. 
Since $M$ is countable,  let $\alpha < \omega_1$ be the supremum of heights of nodes of $M$, and choose a ray $x \in\cR(T')$ of order type $\alpha+\omega+\omega$. Let $x' = x \cap  T^{<\alpha + \omega}$. Then $\cC$ doesn't separate $x$ from $x'$, a contradiction.
\end{proof}

\subsection{Well-behaved dense subspaces in trees without uncountable branches}

Recall that a collection $\Pi$ of nonempty open sets of a space $X$ is a \emph{$\pi$-base} if for every nonempty open set $U \subseteq X$ there is $P \in \Pi$ with $P \subseteq U$.

\begin{prop}
\label{prop_pibase}
The collection $\Pi$ consisting of $\cS=\{[t] \colon t \in T\}$ together with all isolated points forms a $\pi$-base in any ray space $\cR(T)$.
\end{prop}

\begin{proof}
We need to show that every non-empty standard open set $U = [t,F] \in \cC(x)$ includes an element from $\Pi$. If $U$ contains a branch of $T$, then pick such a branch $b \in U \cap \cB(T)$, and note that since $\cS$ forms a base at $b$, there is $t \in T$ with $[t] \subseteq U$, giving us the desired element of the $\pi$-base included in $U$. 

	To complete the proof, we show that if $U \cap \cB(T) = \emptyset$, then $U$ contains an isolated point. %For this, fix $x \in U$ and let $[t,F]$ be a basic open set witnessing $x \in U$. 
Let $y \subseteq x$ be an inclusion-minimal ray with $t \in y$. We claim that $y$ is isolated in $\cR(T)$.

Case 1: $y < x$. Let $s$ be the unique top of $y$ in $x$. Then $[t,\Set{s}]$ witnesses that $y$ is isolated: for any other  ray $z$ in $[t,\Set{s}]$ would extend to a branch $z' \in \cB(T)$ with $z' \in [t,\Set{s}] \subseteq [t,F] \subseteq U$, a contradiction.

Case 2: $y = x$. Then $[t,F]$ witnesses that $y=x$ is isolated: for any other  ray $z$ in $[t,F]$ would extend to a branch $z' \in \cB(T)$ with $z' \in [t,F] \subseteq U$, a contradiction.
\end{proof}

\begin{cor}
	\label{cor_1stctble}
The points of first countability form a dense subspace of any ray space of a tree without uncountable branches.
\end{cor}

\begin{proof}
If we assume $T$ has no uncountable branches, then every $x \in \cB(T)$ is first-countable in $\cR(T)$, for $\{[t] \colon t \in x\}$ is a countable neighbourhood base for $x$. Now since every element of the $\pi$-base from Proposition~\ref{prop_pibase} contains a branch or an isolated point, the density result follows.
\end{proof}

\begin{prop}
\label{prop_densemetr}
Every ray space of a special tree contains a dense, completely 
metrizable subspace. 
\end{prop}

A similar result, that every branch spaces of special tree contains a dense 
metrizable subspace, occurs in Todorcevic's \cite[Theorem~4.2]{todorcevic1981stationary}. Indeed, if $T$ is special, then our $\pi$-base from Proposition~\ref{prop_pibase} will be $\sigma$-disjoint, and a first-countable Hausdorff space has a dense metrizable subspaces if and only if it has a $\sigma$-disjoint $\pi$-base \cite{white1978first}. Our construction requires an additional step, as we aim for completeness as well.

\begin{proof}
For two subsets $V,W \subseteq T$ let us write $V \leq W$ is $V \subseteq \dc{W}_T$, i.e.\ every node in $V$ has a node in $W$ above it.
The conclusion of the Proposition will hold for all trees $T$ for which there is an \emph{(increasing) cofinal antichain} sequence of (maximal) antichains $\set{A_n}:{n \in \NN}$ in $T$ satisfying $A_1 \leq A_2 \leq A_3 \leq \ldots$ and such that for every $t$ there is $a \in \bigcup_{n \in \NN} A_n$ with $t \leq a$. Trees with such an antichain sequence are called \emph{semi-special} in \cite{funk2005branch}. It is routine to show that every special tree is semi-special in this sense.

Now given an increasing cofinal sequence $(A_n)_{n \in \NN}$ consisting of maximal antichains, for every $n \in \NN$ let
\begin{itemize}
    \item $B_n = \set{[t]}:{t \in A_n}$
    \item $X_n = \bigcup B_n \subseteq \cR(T)$
    \item $I_n = \set{x \text{ isolated in } \cR(T)}:{x \notin X_n}$
    \item $X'_n = X_n \cup I_n$
    \item $B'_n = B_n \cup \set{\Set{x}}:{x \in I_n}$.
\end{itemize}
Then $B'_n$ is an open partition of the set $X'_n$, and $X'_n$ is open in $\cR(T)$. Let us see that $X'_n$ is also dense. Let $U \subset \cR(T)$ be non-empty and open. 
If $U$ contains a branch of $T$, then there also is a node $s \in T$ such that $[s] \subset U$. If $\dc{s}$ meets $A_n$, then $[s] \subset B_n$ and $X_n \cap U \neq \emptyset$. Otherwise, it follows from maximality of $A_n$ that there is $t \in A_n \cap [s]$. But then $[t] \subseteq [s]$ for some $t \in A_n$ witnesses $X_n \cap U \neq \emptyset$. 
Otherwise, $U$ avoids $\cB(T)$ and hence contains an isolated point by Proposition~\ref{prop_pibase}, so if $U \cap X_n = \emptyset$, then $U \cap I_n \neq \emptyset$.

Since ray spaces are Baire, Proposition~\ref{prop_basecpt}, it follows that $X = \bigcap_{n \in \NN} X'_n$ is a dense subspace of $\cR(T)$. Moreover, every $B'_n$ induces an open partition of $X$, and $B=\bigcup_{n \in \NN} B'_n$ forms a base for $X$. 
Hence, $B$ is a $\sigma$-discrete base for $X$, so $X$ is metrizable %by the Bing metrization theorem 
\cite[Theorem~4.4.8]{engelking1989book}. 

Finally, to see that $X$ is completely metrizable, we check that $X$ is base-complete as witnessed by $B$. Indeed, let $\cC \subseteq B$ be an infinite centered subcollection. Then $\cC$ contains at most one element from each $B'_n$. If one of them is an isolated point, we are done. Otherwise, each $C \in \cC$ is of the form $[t_C]$ for some $t_C \in T$ and by centeredness, all elements in $\set{t_C}:{C \in \cC}$ are pairwise comparable, so lie on a branch $x$ of $T$. It remains to show that $x \in X$. Otherwise, there is some antichain $A_n$ avoiding $x$. Since $\cC$ is infinite, there must be $m > n$ such that some $t_C \in A_m$. Since $A_n \leq A_m$, however, the above observation implies that there must by an element of $A_n$ below $t_C$ on $x$, a contradiction.
Thus, $x \in X$, and so $X$ is metrizable and base-complete, hence completely metrizable \cite{de1963subcompactness}.
\end{proof}

\section{Characterisations of path-, ray- and branch spaces}

\subsection{Base properties II}
\label{sec_baseII}

Let $(X,\tau)$ be a topological space and $\cA$ a collection of subsets of $X$. 
\begin{itemize}
\item $\cA$ is a \emph{clopen base} for $X$ if all elements of $\cA$ are clopen, and $\cA$ is a base for $X$.
\item $\cA^\complement := \set{A^\complement}:{A \in \cA}$,
		
	\item $\cA$ is a \emph{clopen subbase} for $(X,\tau)$ if $\tau$ is the smallest topology on $X$ containing all elements of $\cA$ as clopen subsets. Equivalently, $\cA$ is a \emph{clopen subbase} if %all elements of $\cA$ are clopen and 
	$\cA \cup \cA^\complement$ is an open subbase for $X$. % in the usual sense.
	\item $\cA$ is \emph{disjoint} if all elements of $\cA$ are pairwise disjoint.
	\item $\cA$ is \emph{discrete} if for every $x \in X$ there is an open neighbourhood of $x$ that intersects at most one element of $\cA$.
	\item $\cA$ is \emph{locally finite} if for every $x \in X$ there is an open neighbourhood of $x$ that intersects only finitely many elements of $\cA$.
	\item $\cA$ is \emph{$\sigma$-disjoint} / \emph{$\sigma$-discrete} / \emph{$\sigma$-locally finite} if $\cA $ is a countable union of disjoint / discrete / locally finite subcollections.
		\end{itemize}
	%These properties are used in topological characterisations of (ultra-)metrizable spaces: By the Nagata-Smirnov-Bing metrization theorem, a Hausdorff space is metrizable if and only if it admits a $\sigma$-discrete (equivalently: $\sigma$-locally finite) base (equivalently: subbase), and by the de Groot metrization theorem \cite{de1956non},  a Hausdorff space is ultrametrizable if and only if it admits a $\sigma$-locally finite clopen (sub-)base.

\subsection{The metrizable case} Since end-, branch-, and ray spaces are base-complete and ultraparacompact (Sections~\ref{sec25} and \ref{sec26}), these classes of spaces -- when metrizable -- describe exactly the class of completely ultrametrizable spaces, property (5) below, which 
can be characterized topologically by the existence of certain well-behaved bases.

\begin{thm} \label{thm_charmetr} The following are equivalent for a Hausdorff space $X$:
	\begin{enumerate}
		\item $X$ is homeomorphic to $\Omega(T) = \cR(T) = \cB(T)$ for a (pruned)  graph-theoretic tree $T$.
		\item $X$ is homeomorphic to a metrizable end space of a graph.
		\item $X$ is homeomorphic to a metrizable ray space.
		\item $X$ is homeomorphic to a metrizable branch space.
		\item $X$ is an ultraparacompact, subcompact, metrizable space.
		\item $X$ is completely ultrametrizable.
		\item $X$ admits a nested, noetherian base that is complete and $\sigma$-discrete.
		\item $X$ admits a nested base that is complete and $\sigma$-discrete.
		\item $X$ admits a clopen base that is complete and $\sigma$-locally finite. 
				\end{enumerate}
\end{thm}

\begin{proof}
	We first establish the metric equivalences from (1)--(6). Trivially, (1) implies (2), (3) and (4). By Propositions~\ref{prop_ultraparacompactness} and \ref{prop_basecpt2}, any of (2), (3) and (4) implies (5). Now every base-complete metrizable space is completely metrizable \cite{de1963subcompactness}, and any completely metrizable, ultraparacompact space is stongly zero-dimensional, and hence completely ultrametrizable by \cite[Corollary~5]{lemin2003ultrametrization}.
		Finally, that every completely ultrametrizable space can be respresented as end space of a graph-theoretic tree is folklore, see e.g.\ \cite{hughes2004trees}, giving (6) $\Rightarrow$ (1).

		Now for the topological equivalences: To see that (1) implies (7), note that for a graph theoretic tree $T$ with levels $T^n$ ($n \in \NN$), the collection $\bigcup_{n \in \NN} \set{[t]}:{t \in T^n}$ forms a base as required in (7). The implications (7) through (9) are trivial. 
		Finally, (9) $\Rightarrow$ (6) is due to de Groot \cite{de1956non}.
\end{proof}

\subsection{Order trees from a nested, noetherian subbase}

As a preparation for our topological characterisations of arbitrary (non-metrizable) ray- and branch spaces, we investigate certain families of open sets that form a tree under reverse inclusion. 

 	\begin{defn}
 	Let $\cS \subset 2^X$ be a collection of subsets of a space $X$.
	An \emph{$\cS$-tree} $(T,f)$ is a tree $T$ together with a surjection $f \colon T \to \cS$ such that 
		\begin{enumerate}
\item[(F1)] if $t \leq t'$, then $f(t) \supseteq f(t')$,  and
\item[(F2)] if $t, t' \in T$ are incomparable, then $f(t) \cap f(t') = \emptyset$.
\end{enumerate}
\end{defn}

\begin{lemma}
\label{lem_treeassociated}
	Let $\cS \subset 2^X$ be a nested, noetherian collection of subsets of $X$ such that $X \in \cS$ but $\emptyset \notin \cS$. Then ordering $\cS$ by reverse inclusion gives rise to a rooted tree $(T_\cS, \leq) = (\cS,\supseteq)$, which is special if and only if $\cS$ is $\sigma$-disjoint.
\end{lemma}

\begin{proof}	
Since $\cS$ is noetherian, it follows that $(\cS,\supseteq)$ is a well-founded partial order, with unique minimal element $X$. So to see that $T_\cS$ is a tree, it remains to show that $\dc{S}$ is a chain for all $S \in T_\cS$. So consider $S',S''$ in $\cS$ with $S',S'' \supseteq S \neq \emptyset$. Since $S',S''$ intersect, it follows from nestedness that one of them is contained in the other.

In fact, this argument shows that incomparable elements in $T_\cS$ are in fact disjoint, showing that $T_\cS$ is special if and only if $\cS$ is $\sigma$-discrete.
\end{proof}

It will be convenient to think of the tree $T_\cS$ as an abstract tree $(T_\cS,\leq)$ together with a bijection $f \colon T_\cS \to \cS$ recording which sets of $\cS$ correspond to which nodes of the tree.

\begin{lemma}
\label{lem_assmap}
Let $X$ be a Hausdorff space and $\cS$ a clopen subbase of a space $X$. Then for every $\cS$-tree $(T,f)$, 
 the function $e \colon X \to \cP(T)$ defined by
$$x\mapsto e(x):= \set{t \in T}:{x \in f(t)}$$
 is an embedding of $X$ into the path space $\cP(T)$.
\end{lemma}

\begin{proof}
To see that $e$ is well-defined, note that (F2) implies that $e(x)$ is a chain, and (F1) implies that $e(x)$ is down-closed. To see that $e$ is continuous, note that for every node $t \in T$, we have $e^{-1}([t]) = f(t) \in \cS$ is (subbasic) clopen in $X$. 	

To see that $e \colon X \to \cP(T)$ is injective, consider $x \neq y \in X$. Since $X$ is Hausdorff, there is $S \in \cS$ such that say $x \in S$ and $y \notin S$. Since $f$ is onto, there is $t \in T$ with $f(t) = S$. Then $t  \in e(x) \setminus e(y)$, witnessing that $e(x)$ and $e(y)$ are distinct paths of $T$. 

To see that $e \colon X \to \cP(T)$ is open onto its image, consider some subbasic clopen $S \in \cS$. We need to show that $e(S)$ is clopen in $e(X) \subseteq \cP(T)$. Since $f$ is onto, there is $t \in T$ with $f(t) = S$. We claim that $e(S) = e(X) \cap [t]$. For $\subseteq$, note that for every $x\in S$ we have $e(x) \ni t$ so $e(x) \in [t]$. For $\supseteq$, let $x \in X$ with $e(x) \in [t]$. By definition of $e$, this means $x \in f(t) = S$.
\end{proof}

\begin{thm}
	A topological space $X$ embeds into some path space if (and only 
	if) $X$ is Hausdorff and has a noetherian, nested clopen subbase.
	\end{thm}
	
	\begin{proof}
		For the if-direction, combine Lemmas~\ref{lem_treeassociated} and \ref{lem_assmap}. The only if-direction will follow from Theorem~\ref{thm_charpath} below, showing that every path-space is a ray-space, and hence has a noetherian, nested clopen subbase (Proposition~\ref{prop_increasing}).
	\end{proof}

\subsection{Ray spaces}
We now prove our characterisation for (subspaces) of ray spaces announced in the introduction, merged for convenience into one single statement.

\begin{thm}
\label{thm_charrayrepeat}
	The following are equivalent for a Hausdorff space $X$:
	\begin{enumerate}
		\item There is a \{surjective\} embedding $X \hookrightarrow \cR(T)$ into the ray space of a [special] tree $T$, 
		\item $X$ admits a  \{hereditarily complete\} nested clopen subbase that is
noetherian [and $\sigma$-disjoint].
 	\end{enumerate}
\end{thm}
This theorem should be read as follows: Adding one or both conditions in brackets in assertion (1), one also needs to add the condition(s) in the corresponding brackets in assertion (2).

\begin{proof}
The implication $(1) \Rightarrow (2)$ follows from the results in Section~\ref{sec25}.

	For $(2) \Rightarrow (1)$, let $\cS$ be a nested noetherian clopen subbase for a Hausdorff space $X$, and consider the associated $\cS$-tree $T' = T_\cS$ from Lemma~\ref{lem_treeassociated} with the embedding
	$$e'\colon X \to \cP(T')$$
	from Lemma~\ref{lem_assmap}.
	Our plan is to modify $T'$ to an $\cS$-tree $T$ such that the resulting $e\colon X \to \cP(T)$ becomes an embedding 
	%$$e \colon X \to \cR(T)$$
	of $X$ into $\cR(T)$
	that will be surjective in the case where $\cS$ is hereditarily complete.

Recall the \emph{lexicographical sum} $\cL (P_t \colon t \in T)$ of a family of posets $(P_t \colon t \in T)$ indexed by another poset $T$ is the set
$$\{(t,p) \colon t \in T, p \in P_t \}$$
provided with the partial order
$$(t,p) \leq^* (t',p') \; \text{ if and only if } t<t',  \text{ or }  (t=t' \text{ and } p \leq p').$$ 
If $T$ is a tree and all $P_t$ are well-ordered, then $\cL (P_t \colon t \in T)$ is a tree, too.

Now let $(T',\leq) = (\cS,\supseteq)$, interpreted as an $\cS$-tree $(T',f')$. For every $t \in T'$ let $P_t = 
	\NN$ if there is $x \in X$ with $e'(x) = \dc{t}$, and let $P_t =\Set{0}$ otherwise. Consider $T = \cL(P_t \colon t \in T')$ as an $\cS$-tree where $f((t,p)) = f'(t)$. Intuitively, this corresponds to inserting an $\omega$-sequence of nodes above certain nodes $t$ of $T'$, with all successors $s$ of $t$ in $T'$ becoming tops of this $\omega$-sequence in $T$.
	
	We claim that the associated embedding $e \colon X \to \cP(T)$ (Lemma~\ref{lem_assmap}) satisfies $e(X) \subseteq \cR(T)$. Indeed, if $e'(x) \in \cR(T')$, then $e(x) \in \cR(T)$ is still true. And if $e'(x)$ is a path with maximum say $t \in T'$, then $\set{(t,n)}:{n \in \NN}$ is cofinal in $e(x)$ by construction, so $e(x)$ is a ray.

Moreover, if $\cS$ is $\sigma$-disjoint, then $T'$ is special by Lemma~\ref{lem_treeassociated}, and then $T$ is special, too.

To complete the proof, it hence remains to show that in the case where $\cS$ is hereditarily complete, the embedding $e$ also satisfies $e(X) \supseteq \cR(T)$.  As first step, we shall argue that if $\cS$ is hereditarily complete, then $\cR(T')$ is contained in the image of $e'$. Indeed, suppose for a contradiction that there is $\varrho' \in \cR(T')$ which is not in the image of $e'$. Write $Z$ for the collection of tops of $\varrho'$ in $T'$ (possibly empty), and let $U = \bigcup_{z \in Z} f'(z)$, an open subset of $X$. Then $U^\complement$ is a closed subspace of $X$. Since $f'(t) \subsetneq f'(t')$ for all $t < t' \in \varrho$, the collection 
$$\set{f'(t) \cap U^\complement}:{t \in \varrho'}$$ has the finite intersection property with empty intersection, contradicting that $\cS$ was hereditarily complete. 

But then $\cR(T)$ is contained in the image of $e$ as well: Indeed, given a ray $\varrho \in \cR(T)$, either $\varrho' = \varrho \cap T'$ is a ray in $T'$, in which case the $x \in X$ satisfying $e'(x) = \varrho '$ also satisfies $e(x) = \varrho$, or $\varrho' = \varrho \cap T'$ is of the form $\varrho' = \dc{t}$, and so $\varrho$ was created in response to the fact that there was an $x \in X$ with $e'(x) = \dc{t}$. But then $e(x) = \varrho$ as desired.
\end{proof}

Observing that all characterising conditions in Theorem~\ref{thm_charrayrepeat} are closed-hereditary, we obtain:

\begin{cor}
\label{cor_closedclosed}
	Every closed subspace of a ray space is itself a ray space.
	\end{cor}

\begin{cor}
	The class of ray spaces and the class of branch spaces are incomparable.
	\end{cor}

\begin{proof}
	The one-point compactification of an uncountable discrete space, Example~\ref{example_onepointcpt}, is a ray space that cannot be represented as a branch space (as every compact branch space is metrizable by a result of Archangel'ski \cite{archangelskii}, see also Nyikos' survey article \cite[Theorem~1.3]{nyikos1975some}.).
	
	Conversely, the branch space $\cB(T)$ of the Michael-line type tree in Example~\ref{example_Michaelline} is not a ray space by Corollary~\ref{cor_closedclosed}, as its closed copy $\QQ$ cannot by represented as a ray space by Theorem~\ref{thm_charmetr}.
		\end{proof}

\subsection{Path spaces}
For our topological characterisation of path spaces, we will show that the class of path spaces coincides with the class of compact ray spaces. Via Theorem~\ref{thm_charrayrepeat}, this gives a purely topological characterisation of path spaces of (special) trees, Corollary~\ref{cor_char_pathspaces} below.

\begin{thm}
\label{thm_charpath}
	The following are equivalent for a space $X$:
	\begin{enumerate}
		\item $X$ is homeomorphic to a path space of a [special] tree,
		\item $X$ is homeomorphic to a compact ray space of a [special] tree.
	\end{enumerate}
\end{thm}

For example, $\cR(2^{< \omega})$ and $\cP(\omega^{<\omega})$ are both crowded, zero-dimensional, metrizable spaces, so homeomorphic to a Cantor set.

\begin{proof}
	For the implication $(1) \Rightarrow (2)$, given a path space $\cP(T)$ of a tree $(T,\leq)$, consider the ray space $\cR(T^*)$ where $T^* = T \times \NN$ with the lexicographic ordering $\leq^*$, i.e. we have 
	$$(t,n) \leq^* (t',n') \; \text{ if and only if } t<t',  \text{ or }  (t=t' \text{ and } n \leq n').$$ 
	Intuitively, this corresponds to inserting an $\omega$-sequence of nodes above every node $t$ of $T$, with all successors $s$ of $t$ in $T$ becoming tops of this $\omega$-sequence.
	
	We claim that 
	$$f \colon \cP(T) \to \cR(T^*), \; p \mapsto p \times \NN$$
	is a homeomorphism. It is clearly a bijection, so since the domain is compact and the range is Hausdorff, it suffices to check for continuity. For this, we argue that preimages of subbasic clopen sets are (subbasic) clopen. Indeed, consider a set $[(t,n)] \in \cR(T^*)$. Since $[(t,n)] = [(t,0)]$ it follows that 
	$$f^{-1}\p{[(t,n)]} =  f^{-1}\p{[(t,0)]} = [t].$$ % and similarly, 
	%$$f^{-1}\p{[(t,n)]^\complement} =  f^{-1}\p{[(t,0)]^\complement} = [t]^\complement.$$
	This shows that $f$ is a homeomorphism. Moreover, if $T$ is special, say $T = \bigcup_{n \in \NN} A_n$ for antichains $A_n$, then also $T^* = \bigcup_{n,m \in \NN} A_n \times \Set{m}$ is a countable union of antichains $A_n \times \Set{m}$, completing the proof of the first implication.
	
	For the implication $(2) \Rightarrow (1)$, consider a compact ray space $\cR(T)$. We may assume that $T$ is pruned. By a transfinite greedy construction, it is straightforward to decompose $T$ into a family $\cI$ of pairwise disjoint intervals in $T$ each of order type $\omega$: at each step, pick a minimal node $t$ of the tree not yet covered by such an interval, and since $T$ is pruned, we may greedily pick an increasing $\omega$ sequence starting in $t$.
	
	Then consider $S = \set{\min I}:{I \in \cI} \subset T$ as a tree with the partial order induced by $T$. In particular, if $T$ is special, then so is $S$. We claim that 
	$$f \colon \cR(T) \to \cP(S), \; x \mapsto x \cap S$$
	is a homeomorphism. The map $f$ is well-defined: if $x$ is a down-closed chain in $T$, then $x \cap S$ is a downclosed chain in $S$, so a path. 
	
	The map $f$ is injective: if $x \neq x'$ are distinct rays of $T$, then either say $x \subseteq x'$ or $x$ and $x'$ are incomparable. In the first case, fix $t \in x' \setminus x$, and let $I \in \cI$ with $t \in I$. Because $I$ is an interval of order type $\omega$, it follows that $I \cap x = \emptyset$, so $\min I \in f(x') \setminus f(x)$ witnesses that $f(x) \neq f(x')$. In the second case, let $t \in x \setminus x'$ and $I \in \cI$ with $t \in I$. Then $I \subset x$. If $I \cap x' = \emptyset$, we are done. Otherwise, pick $t' \in x' \setminus x$, and $I'\in\cI$ with $t' \in I'$. Since $I \cap I' = \emptyset$, we have that $\min I' \in f(x') \setminus f(x)$ witnesses that $f(x) \neq f(x')$. 
	
	The map $f$ is surjective: Let $p \in \cP(S)$. If $p$ has a maximum $s$, consider $I \in \cI$ with $s \in I$ and note that $x = \dc{I} \subset T$ is a ray in $T$ with $f(x) = p$. If $p$ has no maximal element, then $x = \dc{p} \subset T$  is a ray in $T$ with $f(x) = p$.
	
	Finally, since $\cR(T)$ is compact by assumption and $\cP(S)$ is Hausdorff, it suffices to check for continuity. For this, we argue that preimages of subbasic clopen sets are (subbasic) clopen. Indeed, consider a set $U=[s] \subset \cP(S)$. Then $s \in S \subset T$ also induces a subbasic clopen set $V=[s] \subset \cR(T)$, and since the elements in $\cI$ are disjoint intervals, it follows that $f^{-1}(U) = V$. The proof is complete. 
\end{proof}

\begin{cor}
\label{cor_char_pathspaces}
	A topological space is (homeomorphic to) a path space if and only if it is compact Hausdorff and has a noetherian, nested clopen subbase. \hfill \qed
\end{cor}

\begin{cor}
\label{cor_endspaceEberlein}
	The following are equivalent for a Hausdorff space $X$:
	\begin{enumerate}
		\item $X$ is homeomorphic to a compact end space of a graph,
		\item $X$ is homeomorphic to a compact ray space of a special tree,
		\item $X$ is homeomorphic to a path space of a special tree,
		\item $X$ is homeomorphic to an Eberlein compact path space,
		\item $X$ is homeomorphic to an Eberlein compact ray space,
		\item $X$ is Eberlein compact and admits a nested clopen noetherian subbase.
	\end{enumerate}
\end{cor}

\begin{proof}
	The equivalence $(1) \Leftrightarrow (2)$ follows from Theorem~\ref{thm_kp_rep}. The equivalences $(2) \Leftrightarrow (3)$ and $(4) \Leftrightarrow (5)$ follow from Theorem~\ref{thm_charpath}. The equivalence $(3) \Leftrightarrow (4)$ has been proven by Gruenhage in \cite{Gruenhage1986}. And finally, the equivalence $(5) \Leftrightarrow (6)$ follows from Theorem~\ref{thm_charrayrepeat}, noting that in a compact space, any clopen (sub-)base is hereditarily complete.
\end{proof}

\subsection{Branch spaces} In this section we establish the following characterisation of branch spaces.

\begin{thm}
\label{thm_charbranch}
	The following are equivalent for a Hausdorff space $X$:
	\begin{enumerate}
		\item There is a \{surjective\} embedding $X \hookrightarrow \cB(T)$ into the branch space of a [special] tree $T$, 
		\item $X$ admits a  \{complete\} nested 
		base [that is
 $\sigma$-disjoint].
 	\end{enumerate}
\end{thm}

Note that the elements of a nested base are necessarily clopen. 

\begin{proof}
	The forwards implications are witnessed by the base $\cS=\set{[t]}:{t \in T}$ for the branch space $\cB(T)$. 
	
	The asserted characterisation of subspaces of branch spaces, i.e.\ that  Hausdorff space embeds into a branch spaces of a tree if and only if $X$ admits a nested base, is a classic result, with a modern proof provided by P. Nyikos in \cite[Theorem~2.10]{nyikos1999some}. This gives  Theorem~\ref{thm_charbranch} without any of the extra properties in brackets. Nyikos writes that ``\emph{there seems to be no convenient topological characterisation of the branch spaces of trees}'', \cite[p.8]{nyikos1999some} -- but all that is missing for such a  characterisation is a suitable completeness assumption in spirit of the Amsterdam properties. 
	
	Indeed, in \cite[Theorem~2.9 \& 2.10]{nyikos1999some} Nyikos constructs from a nested base $\cS \not\ni \emptyset$ a base $\cS' \subset \cS$ and $\cS'$-tree $T$ such that $t< t'$ implies $f(t) \supsetneq f(t')$ and the associated map $e \colon X \to \cB(T)$ defined by
$$e \colon X \to \cB(T), \; x\mapsto e(x):= \set{t \in T}:{x \in f(t)}$$ is an embedding. Now it suffices to observe that if $\cS$ is complete, then $e \colon X \to \cB(T)$ is onto: for any branch $b \in \cB(T)$, the collection $\set{f(t)}:{t \in b} \subseteq \cS$ is centered, and the element $x \in \bigcap_{t \in b} f(t)$ satisfies $e(x) = b$.

Finally, if $\cS$ is $\sigma$-disjoint, then from the fact that $t< t'$ implies $f(t) \supsetneq f(t')$, it readily follows that $T$ is special.
\end{proof}

\bibliographystyle{plain}
\bibliography{ref}

\begin{bibdiv}
\begin{biblist}

\bib{completeness}{article}{
      author={Aarts, J.},
      author={Lutzer, D.},
       title={Completeness properties designed for recognizing {B}aire spaces},
        date={1974},
     journal={Dissertationes Mathematicae},
      volume={116},
       pages={1\ndash 45},
}

\bib{archangelskii}{article}{
      author={Archangel'skii, A.V.},
       title={Ranks of systems of sets and dimensionality of spaces (in
  {R}ussian)},
        date={1963},
     journal={Fundamenta Mathematicae},
      volume={52},
      number={3},
       pages={257\ndash 275},
}

\bib{carmesin2019all}{article}{
      author={Carmesin, Johannes},
       title={All graphs have tree-decompositions displaying their topological
  ends},
        date={2019},
     journal={Combinatorica},
      volume={39},
      number={3},
       pages={545\ndash 596},
}

\bib{cheriyan19992}{inproceedings}{
      author={Cheriyan, Joseph},
      author={Jord{\'a}n, Tibor},
      author={Ravi, Ramamoorthi},
       title={On 2-coverings and 2-packings of laminar families},
organization={Springer},
        date={1999},
   booktitle={European symposium on algorithms},
       pages={510\ndash 520},
}

\bib{de1956non}{article}{
      author={De~Groot, J.},
       title={Non-archimedean metrics in topology},
        date={1956},
     journal={Proceedings of the American Mathematical Society},
      volume={7},
      number={5},
       pages={948\ndash 953},
}

\bib{de1963subcompactness}{article}{
      author={De~Groot, J.},
       title={Subcompactness and the {B}aire category theorem},
        date={1963},
     journal={Indag. Math},
      volume={22},
       pages={761\ndash 767},
}

\bib{diestel1992end}{article}{
      author={Diestel, R.},
       title={The end structure of a graph: recent results and open problems},
        date={1992},
     journal={Discrete Mathematics},
      volume={100},
      number={1--3},
       pages={313\ndash 327},
}

\bib{diestel2006end}{article}{
      author={Diestel, R.},
       title={End spaces and spanning trees},
        date={2006},
     journal={Journal of Combinatorial Theory, Series B},
      volume={96},
      number={6},
       pages={846\ndash 854},
}

\bib{diestel2018tree}{article}{
      author={Diestel, R.},
       title={Tree sets},
        date={2018},
     journal={Order},
      volume={35},
      number={1},
       pages={171\ndash 192},
}

\bib{diestel2003graph}{article}{
      author={Diestel, R.},
      author={K{\"u}hn, D.},
       title={Graph-theoretical versus topological ends of graphs},
        date={2003},
     journal={Journal of Combinatorial Theory, Series B},
      volume={87},
      number={1},
       pages={197\ndash 206},
}

\bib{engelking1989book}{book}{
      author={Engelking, Ryszard},
       title={General toplogy -- revised and completed ed.},
   publisher={Heldermann Verlag Berlin},
        date={1989},
      volume={6},
}

\bib{freudenthal1942neuaufbau}{article}{
      author={Freudenthal, Hans},
       title={Neuaufbau der {E}ndentheorie},
        date={1942},
     journal={Annals of Mathematics},
      volume={43},
      number={2},
       pages={261\ndash 279},
}

\bib{freudenthal1944enden}{article}{
      author={Freudenthal, Hans},
       title={{\"U}ber die {E}nden diskreter {R}{\"a}ume und {G}ruppen},
        date={1944},
     journal={Commentarii Math.~Helvetici},
      volume={17},
      number={1},
       pages={1\ndash 38},
}

\bib{funk2005branch}{article}{
      author={Funk, Will},
      author={Lutzer, David},
       title={Branch space representations of lines},
        date={2005},
     journal={Topology and its Applications},
      volume={151},
      number={1-3},
       pages={187\ndash 214},
}

\bib{gollin2021representations}{article}{
      author={Gollin, J~Pascal},
      author={Kneip, Jakob},
       title={Representations of infinite tree sets},
        date={2021},
     journal={Order},
      volume={38},
      number={1},
       pages={79\ndash 96},
}

\bib{Gruenhage1986}{article}{
      author={Gruenhage, Gary},
       title={On a {C}orson space of {T}odor{\v{c}}evi{\'c}},
        date={1986},
     journal={Fundamenta Mathematicae},
      volume={126},
      number={3},
       pages={261\ndash 268},
         url={http://eudml.org/doc/211602},
}

\bib{Halin_Enden64}{article}{
      author={Halin, R.},
       title={{Über unendliche Wege in Graphen}},
        date={1964},
     journal={Mathematische Annalen},
      volume={157},
       pages={125\ndash 137},
}

\bib{halin1978simplicial}{incollection}{
      author={Halin, R.},
       title={Simplicial decompositions of infinite graphs},
        date={1978},
   booktitle={Advances in {G}raph {T}heory, {A}nnals of {D}iscrete
  {M}athematics},
      editor={Bollob\'as, B.},
      volume={3},
   publisher={North-Holland},
}

\bib{hopf1943enden}{article}{
      author={Hopf, Heinz},
       title={Enden offener {R}{\"a}ume und unendliche diskontinuierliche
  {G}ruppen},
        date={1943},
     journal={Commentarii Mathematici Helvetici},
      volume={16},
      number={1},
       pages={81\ndash 100},
}

\bib{hughes2004trees}{article}{
      author={Hughes, Bruce},
       title={Trees and ultrametric spaces: a categorical equivalence},
        date={2004},
     journal={Advances in Mathematics},
      volume={189},
      number={1},
       pages={148\ndash 191},
}

\bib{jung1968wurzelbaume}{article}{
      author={Jung, H.~A.},
       title={Wurzelb{\"a}ume und {K}antenorientierungen in {G}raphen},
        date={1968},
     journal={Mathematische Nachrichten},
      volume={36},
      number={5--6},
       pages={351\ndash 359},
}

\bib{kurepa1956ecart}{article}{
      author={Kurepa, Dj},
       title={Sur l’{\'e}cart abstrait},
        date={1956},
     journal={Glasnik mat. fiz. astr},
      volume={11},
       pages={105\ndash 134},
}

\bib{kurkofkapitz_rep}{misc}{
      author={Kurkofka, J.},
      author={Pitz, M.},
       title={A representation theorem for end spaces},
        date={(2021)},
        note={Submitted.},
}

\bib{FirstSecondCountable}{article}{
      author={Kurkofka, Jan},
      author={Melcher, Ruben},
       title={Countably determined ends and graphs},
        date={2022},
     journal={Journal of Combinatorial Theory, Series B},
      volume={156},
       pages={31\ndash 56},
}

\bib{kurkofka2021approximating}{article}{
      author={Kurkofka, Jan},
      author={Melcher, Ruben},
      author={Pitz, Max},
       title={Approximating infinite graphs by normal trees},
        date={2021},
     journal={Journal of Combinatorial Theory, Series B},
      volume={148},
       pages={173\ndash 183},
}

\bib{lemin2003ultrametrization}{article}{
      author={Lemin, Alex},
       title={On ultrametrization of general metric spaces},
        date={2003},
     journal={Proceedings of the American Mathematical Society},
      volume={131},
      number={3},
       pages={979\ndash 989},
}

\bib{nyikos1975some}{incollection}{
      author={Nyikos, Peter~J.},
       title={Some surprising base properties in topology},
        date={1975},
   booktitle={Studies in topology},
   publisher={Elsevier},
       pages={427\ndash 450},
}

\bib{nyikos1999some}{article}{
      author={Nyikos, Peter~J.},
       title={On some non-{A}rchimedean spaces of {A}lexandroff and {U}rysohn},
        date={1999},
     journal={Topology and its Applications},
      volume={91},
      number={1},
       pages={1\ndash 23},
}

\bib{pitz2020unified}{article}{
      author={Pitz, Max},
       title={A unified existence theorem for normal spanning trees},
        date={2020},
     journal={Journal of Combinatorial Theory, Series B},
      volume={145},
       pages={466\ndash 469},
}

\bib{polat1996ends}{article}{
      author={Polat, Norbert},
       title={Ends and multi-endings, {I}},
        date={1996},
     journal={Journal of Combinatorial Theory, Series B},
      volume={67},
       pages={86\ndash 110},
}

\bib{polat1996ends2}{article}{
      author={Polat, Norbert},
       title={Ends and multi-endings, {II}},
        date={1996},
     journal={Journal of Combinatorial Theory, Series B},
      volume={68},
       pages={56\ndash 86},
}

\bib{sierpinski1920propriete}{article}{
      author={Sierpi{\'n}ski, Wac{\l}aw},
       title={Sur une propri{\'e}t{\'e} topologique des ensembles
  d{\'e}nombrables denses en soi},
        date={1920},
     journal={Fundamenta Mathematicae},
      volume={1},
      number={1},
       pages={11\ndash 16},
}

\bib{sprussel2008end}{article}{
      author={Spr{\"u}ssel, P.},
       title={End spaces of graphs are normal},
        date={2008},
     journal={Journal of Combinatorial Theory, Series B},
      volume={98},
      number={4},
       pages={798\ndash 804},
}

\bib{todorcevic1981stationary}{article}{
      author={Todor{\v{c}}evi{\'c}, Stevo},
       title={Stationary sets, trees and continuums},
        date={1981},
     journal={Publ. Inst. Math.(Beograd)(NS)},
      volume={29},
      number={43},
       pages={249\ndash 262},
}

\bib{todorvcevic1984trees}{incollection}{
      author={Todor{\v{c}}evi{\'c}, Stevo},
       title={Trees and linearly ordered sets},
        date={1984},
   booktitle={Handbook of set-theoretic topology},
   publisher={Elsevier},
       pages={235\ndash 293},
}

\bib{todorvcevic1988lindelof}{article}{
      author={Todor{\v{c}}evi{\'c}, Stevo},
       title={On the {L}indel{\"o}f property of {A}ronszajn trees},
        date={1988},
     journal={General Topology and Its Relations to Modern Analysis and
  Algebra, VI (Praha, 1986), Heldermann, Berlin},
       pages={577\ndash 588},
}

\bib{van1979subbase}{article}{
      author={van Mill, Jan},
      author={Schrijver, Alexander},
       title={Subbase characterizations of compact topological spaces},
        date={1979},
     journal={General Topology and Its Applications},
      volume={10},
      number={2},
       pages={183\ndash 201},
}

\bib{white1978first}{article}{
      author={White, HE},
       title={First countable spaces that have special pseudo-bases},
        date={1978},
     journal={Canadian Mathematical Bulletin},
      volume={21},
      number={1},
       pages={103\ndash 104},
}

\end{biblist}
\end{bibdiv}
\end{document}